\documentclass[10pt]{article}
\usepackage[dvips]{graphicx}
\usepackage{amssymb,color}
\usepackage[latin1]{inputenc}
\usepackage{epic}
\usepackage{pstricks}
\usepackage{pst-node}
\usepackage{pstricks-add}
\usepackage[latin1]{inputenc}
\usepackage{color}
\usepackage{lettrine}
\usepackage{calc,pifont}
\usepackage[noindentafter,calcwidth]{titlesec}
\usepackage{pstricks}
\usepackage{pst-plot}
\usepackage{pst-node}
\usepackage{amsmath}
\usepackage{amssymb}
\usepackage{amsthm}
\usepackage{subfigure}
\usepackage{lineno}

\def\BBox{\kern  -0.2cm\hbox{\vrule width 0.2cm height 0.2cm}}

\newtheorem{remark}{Remark}

\newtheorem{teo}{Theorem}[section]
\newtheorem{coro}[teo]{Corollary}
\newtheorem{lema}[teo]{Lemma}

\theoremstyle{definition}

\theoremstyle{remark}

\title{On two-quotient strong starters for $\mathbb{F}_q$}
\author{
Carlos A. Alfaro\thanks{carlos.alfaro@banxico.org.mx}\\{\small  Banco de México}\\[1ex]\\Christian Rubio-Montiel\thanks{christian@cs.cinvestav.mx}
\\{\small UMI LAFMIA}%
\\
{\small CINVESTAV-IPN}\\[1ex]\\
Adrián Vázquez-Ávila\thanks{adrian.vazquez@unaq.edu.mx}\\
{\small Subdirección de Ingeniería y Posgrado}\\
{\small Universidad Aeronáutica en Querétaro}\\
}

\date{}

\begin{document}
\maketitle
\begin{abstract}
Let $G$ be a finite additive abelian group of odd order $n$, and let $G^*=G\setminus\{0\}$ be the set of non-zero elements. A \emph{starter} for $G$ is a set $S=\{\{x_i,y_i\}:i=1,\ldots,\frac{n-1}{2}\}$ such that
\begin{enumerate}
	\item  $\displaystyle\bigcup_{i=1}^{\frac{n-1}{2}}\{x_i,y_i\}=G^*$, and
	\item $\{\pm(x_i-y_i):i=1,\ldots,\frac{n-1}{2}\}=G^*$.
\end{enumerate}
Moreover, if $\left|\left\{x_i+y_i:i=1,\ldots,\frac{n-1}{2}\right\}\right|=\frac{n-1}{2}$, then $S$ is called a \emph{strong starter} for $G$. A starter $S$ for $G$ is a \emph{$k$ quotient starter} if there exists $Q\subseteq G^*$ of cardinality $k$ such that $y_i/x_i\in Q$ or $x_i/y_i\in Q$, for $i=1,\ldots,\frac{n-1}{2}$. In this paper, we give examples of two-quotient strong starters for $\mathbb{F}_q$, where $q=2^kt+1$ is a prime power with $k>1$ a positive integer and
$t$ an odd integer greater than 1.
\end{abstract}


\section{Introduction}

Strong starters were first introduced by Mullin and Stanton in \cite{MR0234587} in constructing of Room squares. Starters and strong starters have been useful to construct many combinatorial designs such as Room cubes \cite{MR633117}, Howell designs \cite{MR728501,MR808085}, Kirkman triple systems \cite{MR808085,MR0314644}, Kirkman squares and cubes \cite{MR833796,MR793636},  and factorizations of complete graphs \cite{MR0364013,MR2206402,MR1010576,MR623318,MR685627}. Moreover, there are some interesting results on strong starters for cyclic groups \cite{MR808085} and for finite abelian groups \cite{MR1010576,MR1044227}.

A starter $S$ is a \emph{$k$ quotient starter} if there exists $Q\subseteq G^*$ of cardinality $k$ such that $y_i/x_i\in Q$ or $x_i/y_i\in Q$, for $i=1,\ldots,\frac{n-1}{2}$, see \cite{dinitz1984}. In particular if $k=1$ the starter $S$ is called \emph{one-quotient} starer for $G$. In fact, an first example of a one-quotient strong starter $S$ was given in \cite{MR0249314} Lemma 1. Further information about quotient starters in \cite{dinitz1984}.

Let $QR(q)$ and $NQR(q)$ denote the set of quadratic residues and the set of non-quadratic residues of the $\mathbb{F}_q^*=\mathbb{F}_q\setminus\{0\}$, respectively. In this work, we prove the following: 

\begin{teo}[Main Theorem]\label{thm:main}
Let $q=2^kt+1$ be a prime power with $k>1$ a positive integer and
$t$ an odd integer greater than 1. Then there exists a strong starter $S$ for $\mathbb{F}_q$ which satisfies
$\{a,b\}\in S$ with $a\in QR(q)$ and $b\in NQR(q)$. Furthermore there exists two different elements $\beta_1,\beta_2\in NQR(q)$ such that for every $\{a,b\}\in S$, with $a\in QR(q)$ and $b\in NQR(q)$, we have that $b/a\in\{\beta_1,\beta_2\}$.
\end{teo}

In the known results, there are constructions of strong starters for $\mathbb{F}_q$ (see \cite{MR0325419,dinitz1984,MR0392622,MR0249314,MR0260604}), but none of those constructions gives an explicit construction of strong starters which satisfy the conclutions of the Main Theorem.

This paper is organized as follows. 
In Section 2, we recall some basic properties about quadratic residues. 
In Section 3, we include an alternative proof when $q\equiv3$ mod 4 and $q\neq3$ (see \cite{MR0249314}). In section 4, we present some previous results. Finally, in Section 5, we prove the main theorem and present some examples.
\section{Quadratic residues}\label{sec:quadratic}
Let $q$ be an odd prime power. An element $x\in\mathbb{F}_q^*$ is called a \emph{quadratic residue} if there exists an element $y\in\mathbb{F}_q^{*}$ such that $y^2=x$. If there is no such $y$, then $x$ is called a \emph{non-quadratic residue.} The set of quadratic residues of $\mathbb{F}_q^{*}$ is denoted by $QR(q)$ and the set of non-quadratic residues is denoted by $NQR(q)$. It is well known that $QR(q)$ is a cyclic subgroup of $\mathbb{F}_q^{*}$ of cardinality $\frac{q-1}{2}$ (see \cite{MR2445243}  pg. 87), that is

\begin{teo}
	Let $q$ be an odd prime power, then $QR(q)$ is a cyclic subgroup of $\mathbb{F}_q^{*}$. Furthermore, $|QR(q)|=|NQR(q)|=\frac{q-1}{2}$. 
\end{teo}

\begin{coro}\label{col:mulres}
	Let $q$ be an odd prime power, then
	\begin{enumerate}
		\item if either $x,y\in QR(q)$ or $x,y\in NQR(q)$, then $xy\in QR(q)$,
		\item if $x\in QR(q)$ and $y\in NQR(q)$, then $xy\in NQR(q)$.
	\end{enumerate}
\end{coro}

The following theorems are well known results on quadratic residues. 
For more details of this kind of results the reader may consult \cite{burton2007elementary} pg. 171, see also \cite{MR2445243}.

\begin{teo}[Eulers' criterion]
	Let $q$ be an odd prime and $x\in\mathbb{F}_q^{*}$, then
	\begin{enumerate}
		\item $x\in QR(q)$ if and only if $x^{\frac{q-1}{2}}=1$.
		\item $x\in NQR(q)$ if and only if $x^{\frac{q-1}{2}}=-1$.
	\end{enumerate}
\end{teo}

\begin{teo}\label{col:menosuno}
	Let $q$ be an odd prime power, then
	\begin{enumerate}
		\item $-1\in QR(q)$ if and only if $q\equiv1$ mod $4$.
		\item $-1\in NQR(q)$ if and only if $q\equiv3$ mod $4$.
	\end{enumerate}
\end{teo}

\begin{teo}\label{inverso}
	Let q be an  odd prime. If $q\equiv1$ mod $4$, then 
	\begin{enumerate}
		\item $x\in QR$ if and only if $-x\in QR$.
		\item $x\in NQR$ if and only if $-x\in NQR$.
	\end{enumerate}
\end{teo}
\section{Case $q\equiv3$ mod 4, with $q\neq3$}
Now, for the sake of completeness, we include a proof for the case $q\equiv3$ mod 4, with $q\neq3$, using different notation that in \cite{MR0249314}. This notation will be used, of general way, in the main result of this paper. 

\begin{lema}
	If $q\equiv3$ mod 4 is an odd prime power with $q\neq3$, then there exists a strong starter $S$ for $\mathbb{F}_q$ such that $\{a,b\}\in S$ satisfy that $a\in QR(q)$ and $b\in NQR(q)$.
\end{lema}
\begin{proof}
	Let $\alpha$ be a generator of $QR(q)$ and $\beta\in NQR(q)$ such that $\beta+1\neq0$. We claim that the following set:
	\begin{eqnarray*}\label{strong_1}
		S_\beta=\left\{\{\alpha,\alpha\beta\},\{\alpha^2,\alpha^2\beta\},\ldots,\{\alpha^\frac{p-1}{2},\alpha^\frac{p-1}{2}\beta\}\right\},
	\end{eqnarray*}
	is a strong starter for $\mathbb{F}_q$. First we have that $\{\alpha,\alpha^2,\ldots,\alpha^{\frac{q-1}{2}}\}=QR(q)$ and  $\{\alpha\beta,\alpha^2\beta,\ldots,\alpha^{\frac{q-1}{2}}\beta\}=\beta QR=NQR(q)$ (by Corollary \ref{col:mulres}). 
	Now we shall prove that $\left\{\pm\alpha^i(\beta-1): i=0,\ldots,\frac{q-1}{2}\right\}=\mathbb{F}_q^*$. 
	Suppose that $\alpha^i(\beta-1)=\pm\alpha^j(\beta-1)$, then $(\beta-1)(\alpha^i\pm\alpha^j)=0$, for $i\neq j\in\{1,\ldots,\frac{q-1}{2}\}$, which it is a contradiction, since if $i<j$ then $\alpha^i(1+\alpha^{j-i})\neq0$, that is, $1+\alpha^{j-i}\neq0$ (by Theorem \ref{col:menosuno}). 
	Finally, we have that $\left|\{\alpha^i(\beta+1): i=0,\ldots,\frac{q-1}{2}\}\right|=\frac{q-1}{2}$, since if $i\neq j\in\{1,\ldots,\frac{q-1}{2}\}$ then $\alpha^i(\beta+1)\neq\alpha^j(\beta+1)$. 
\end{proof}

To end this section, it is easy to see that $S_\beta$ is an example of one quotient strong starter for $\mathbb{F}_q$. This kind of starters are called \emph{Dinitz starters} for $\mathbb{F}_q$, see \cite{book:206537}, Theorem VI.55.22, page 624.

\section{Previous results}\label{sec:particular}
To begin with, we introduce some terminology in order to simplify the description of the of existence of strong starters $S$ for $\mathbb{F}_q$ with the property that if $\{a,b\}\in S$, then $a\in QR(q)$ and $b\in NQR(q)$.

Let $q=2^kt+1$ be a prime power with $k>1$ a positive integer and
$t$ an odd integer greater than 1 and $\alpha$ be a generator of $QR(q)$. We define $C_0=\langle\alpha^\Delta\rangle$, where $\Delta=2^{k-1}$, to be the subgroup of $\mathbb{F}_q^*$ of order $t$. Let
$C_j=\alpha^jC_0$, for $j
=1,\ldots,\Delta_1-1$ with $\Delta_1=2^{k-1}$, and $\hat{C}_j=-C_j$, for $j=0,\ldots,\Delta_1-1$ . Hence $QR(q)=\bigcup_{j=0}^{\Delta_1-1}(C_j\cup\hat{C}_j)$. On the other hand, let $\beta_1\in NQR(q)$ and $\beta_2\in\beta_1\hat{C}_0$. We define $D_j=\beta_1 C_j$ and $\hat{D}_j=\beta_2 C_j$, for $j=0,\ldots,\Delta_1-1$.  Hence $NQR(q)=\bigcup_{i=0}^{\Delta_1-1}(D_j\cup \hat{D}_j)$. Moreover, it is easy to see that $\sum_{a\in C_j}a=0$, $\sum_{a\in \hat{C}_j}a=0$, $\sum_{a\in D_j}a=0$ and $\sum_{a\in \hat{D}_j}a=0$, for $j=1,\ldots,\Delta_1-1$.

To prove the main theorem of this paper (see Theorem \ref{thm:main}), we need to prove the following auxiliary lemma, which states the condition of existence of strong starters $S$ for $\mathbb{F}_q$ with the property that if $\{a,b\}\in S$, then $a\in QR(q)$ and $b\in NQR(q)$, and wthe proof of this lemma is obtained from Lemmas \ref{lemma:completo} and \ref{lemma:(beta+1)(beta-1)}. 

\begin{lema}\label{coro:completo}
Let $q=2^kt+1$ be a prime power with $k>1$ a positive integer and
$t$ an odd integer greater than 1. Then there exist $\beta_1\in NQR(q)$ and $\beta_2\in\beta_1\hat{C}_0$, such that $(\beta_1-1)(\beta_2+1)\in NQR(q)$ and $(\beta_1+1)(\beta_2-1)\in NQR(q)$.
\end{lema}

We present the sketch of the proof of Main Theorem: Let $\alpha$ be a gene\-rator of $QR(q)$, $\beta_1\in NQR(q)$ and $\beta_2\in\beta_1\hat{C}_0$ such that $(\beta_1-1)(\beta_2+1)\in NQR(q)$ and $(\beta_1+1)(\beta_2-1)\in NQR(q)$ (by Lemma \ref{coro:completo}), then the following set
$S(\beta_1,\beta_2)=\displaystyle\bigcup_{i=0}^{\Delta_1-1}S(\beta_1,\beta_2)_j$ is a strong starter for $\mathbb{F}_q$, where
$S(\beta_1,\beta_2)_{j}=\{\{x,\beta_1x\},\{y,-\beta_2y\}:x\in C_j,y\in\hat{C}_j\}$, for $j=0,\ldots,\Delta_1-1$. We have that $QR(q)=\bigcup_{j=0}^{\Delta_1-1}\{\{x\}\cup\{y\}:x\in C_j,y\in\hat{C}_j\}$ and $NQR(q)=\bigcup_{j=0}^{\Delta_1-1}\{\{\beta_1x\}\cup\{-\beta_2y\}:x\in C_j,y\in\hat{C}_j\}$. Moreover, if $\{a,b\}\in S(\beta_1,\beta_2)$, then $a/b\in\{\beta_1,-\beta_2\}$ or $b/a\in\{\beta_1,-\beta_2\}$, which impliest that $S(\beta_1,\beta_2)$ is a two-quotient strong starter for $\mathbb{F}_q$. Moreover, if $\{a,b\}\in S(\beta_1,\beta_2)$, then $a/b\in\{\beta_1,-\beta_2\}$ or $b/a\in\{\beta_1,-\beta_2\}$, which impliest that $S(\beta_1,\beta_2)$ is a two-quotient strong starter for $\mathbb{F}_q$.

To prove the following lemma, we needs the next definition: Let $q=ef+1$ be a prime power and let $H$ be the subgroup of $\mathbb{F}_q^*$ of order $f$ with $\{H=C_0,\ldots,C_{e-1}\}$ the set of (multiplicative) cosets of $H$ in $\mathbb{F}_q^*$ (that is, $C_i = g^iC_0$, where $g$ is the least primitive element of $\mathbb{F}_q$). The cyclotomic number
$(i,j)$ is $|\{x\in C_i: x+1\in C_j\}|$. In particular, if $e=2$ and $f$ is even, then $(0,0)=\frac{f-2}{2}$, $(0,1)=\frac{f}{2}$, $(1,0)=\frac{f}{2}$ and $(1,1)=\frac{f}{2}$, see \cite{book:206537}, Table VII.8.50. Hence if $C_0=QR(q)$ and $C_1=NQR(q)$ are the cosets of $QR(q)$ in $\mathbb{F}_q^*$, then we have the following:

\begin{lema}\label{lemma:1+beta}
Let $q=2^kt+1$ be a prime power with $k>1$ a positive integer and
$t$ an odd integer greater than 1. Then there exist $\beta_1,\beta_2\in NQR(q)$ such that 
\begin{enumerate}
	\item $(\beta_1+1)\in NQR(q)$ and $(\beta_2+1)\in QR(q)$.
	\item $(\beta_1-1)\in NQR(q)$ and $(\beta_2-1)\in QR(q)$.
\end{enumerate}
\end{lema}

\begin{lema}\label{lemma:completo}
Let $q=2^kt+1$ be a prime power with $k>1$ a positive integer and
$t$ an odd integer greater than 1. Then there exists $\beta\in NQR(q)$ such that $(\beta+1)(\beta-1)\in NQR(q)$.
\end{lema}
\begin{proof}
	For each $\beta\in NQR(q)$ define $A_\beta=\{a^\beta_1,\ldots,a^\beta_{l_\beta}\}$, where $a^\beta_{i}\in NQR(q)$ and $a^\beta_{i+1}=a^\beta_{i}+1$, for all $i=1, \ldots,l_{\beta}-1$. 
	By Lemma \ref{lemma:1+beta}, there exists $\beta\in NQR(q)$ such that $|A_\beta|>1$. 
	If $\beta^*=a^\beta_{l_\beta}$, then $(\beta^*+1)\in QR(q)$ and $(\beta^*-1)\in NQR(q)$. 
	On the other hand, if $\beta^*=\beta_1$, then $(\beta^*+1)\in NQR(q)$ and $(\beta^*-1)\in QR(q)$. Hence $(\beta^*+1)(\beta^*-1)\in NQR(q)$.
\end{proof}
\begin{remark}
According with the proof of Lemma \ref{lemma:completo}, if there exists $\beta\in NQR(q)$ such that $|A_\beta|=1$, then $(\beta+1)\in QR(q)$ and $(\beta-1)\in QR(q)$. Moreover, if there exists $\beta\in NQR(q)$ such that $|A_\beta|>2$, then $\beta^*=a^\beta_2$ is such that $(\beta^*+1)\in NQR(q)$ and $(\beta^*-1)\in NQR(q)$.	
\end{remark}

\begin{lema}\label{lemma:(beta+1)(beta-1)}
Let $q=2^kt+1$ be a prime power with $k>1$ a positive integer and
$t$ an odd integer greater than 1. Then given $\beta\in NQR(q)$ there exist $\beta_1,\beta_2\in\beta\hat{C}_0$, such that 	$(\beta_1-1)(\beta_2-1)\in NQR(q)$.
\end{lema}
\begin{proof}
Suppose that given $\beta_1\in\beta\hat{C}_0$ is such that $(\beta_1-1)\in QR(q)$; the same argument is used if we suppose that $(\beta_1-1)\in NQR(q)$. Hence $(\beta_1-1)^{\frac{q-1}{2}}=(\beta_1-1)^{\Delta t}=1$. Let $A_0=\beta\hat{C}_0$, then
	\begin{eqnarray*}
		t&=&\displaystyle\sum_{\beta_1\in A_0}(\beta_1-1)^{\Delta t}=\displaystyle\sum_{x\in C_0}(-\beta x-1)^{\Delta t}\\
		&=&\displaystyle\sum_{x\in C_0}(\beta x+1)^{\Delta t}=\displaystyle\sum_{x\in C_0}\displaystyle\sum_{i=0}^{\Delta t}\binom {\Delta t} {i}\beta^ix^i\\
		&=&\sum_{i=0}^{\Delta t}\binom {\Delta t} {i}\beta^i\displaystyle\sum_{x\in C_j}x^i,
	\end{eqnarray*}
	where
	\begin{equation*}
	\displaystyle\sum_{x\in C_j}x^i=
	\left\{
	\begin{array}{ll}
	t & \hbox{if $i=0,t,\ldots,\Delta t$,}\\
	0 & \hbox{otherwise.}
	\end{array}
	\right.
	\end{equation*}
	
	Then $$1=\sum_{s=0}^{\Delta}\binom {\Delta t} {st}\beta^{st}, \mbox{ for all $\beta\in NQR(q)$}.$$
	
	Hence, if $NQR(q)=\{\beta_1,\ldots,\beta_{\Delta t}\}$, then
	$$\Delta t=\sum_{s=0}^{\Delta}\binom {\Delta t} {st}\sum_{k=1}^{\Delta t}\beta_k^{st}$$ Since
	\begin{equation*}
	\displaystyle\sum_{k=1}^{\Delta t}\beta^{s t}_k=
	\left\{
	\begin{array}{ll}
	\Delta t & \hbox{if $s=0$,}\\
	-\Delta t & \hbox{if $s=\Delta$,}\\
	0 & \hbox{otherwise,}
	\end{array}
	\right.
	\end{equation*}then $$\Delta t=\sum_{s=0}^{\Delta}\binom {\Delta t} {st}\sum_{k=1}^{\Delta t}\beta_k^{st}=0,$$which it is a contradiction. Furthermore, there exist $\beta_1,\beta_2\in \beta\hat{C}_0$, such that $(\beta_1-1)(\beta_2-1)\in NQR(q)$.
\end{proof}

\begin{coro}\label{coro:(beta+1)(beta-1)}
Let $q=2^kt+1$ be a prime power with $k>1$ a positive integer and
$t$ an odd integer greater than 1. Then given $\beta\in NQR(q)$ there exist $\beta_1,\beta_2\in\beta\hat{C}_0$ such that 	$(\beta_1+1)(\beta_2+1)\in NQR(q)$.
\end{coro}	

\section{Proof of the main theorem}\label{sec:main}

With the results presented before, we are ready to prove the main theorem, Theorem \ref{thm:main}, of this paper:

\begin{proof}
Let $\alpha$ be a generator of $QR(q)$, $\beta_1\in NQR(q)$ and $\beta_2\in\beta_1\hat{C}_0$ such that $(\beta_1-1)(\beta_2+1)\in NQR(q)$ and $(\beta_1+1)(\beta_2-1)\in NQR(q)$ (by Lemma \ref{coro:completo}). We claim that the following set
$S(\beta_1,\beta_2)=\bigcup_{i=0}^{\Delta_1-1}S(\beta_1,\beta_2)_j$ is a strong starter for $\mathbb{F}_q$,where $S(\beta_1,\beta_2)_{j}=\{\{x,\beta_1x\},\{y,-\beta_2y\}:x\in C_j,y\in\hat{C}_j\}$, for $j=0,\ldots,\Delta_1-1$. We have that $QR(q)=\bigcup_{j=0}^{\Delta_1-1}\{\{x\}\cup\{y\}:x\in C_j,y\in\hat{C}_j\}$ and $NQR(q)=\bigcup_{j=0}^{\Delta_1-1}\{\{\beta_1x\}\cup\{-\beta_2y\}:x\in C_j,y\in\hat{C}_j\}$. Moreover, if $\{a,b\}\in S(\beta_1,\beta_2)$, then $a/b\in\{\beta_1,-\beta_2\}$ or $b/a\in\{\beta_1,-\beta_2\}$, which impliest that $S(\beta_1,\beta_2)$ is a two-quotient strong starter for $\mathbb{F}_q$. 

First we shall prove that $\mathbb{F}_q^*=\bigcup_{j=0}^{\Delta_1-1} (E_j\cup E^*_j)$, where
	\begin{center}
		$E_j=\left\{\pm x(\beta_1-1):x\in C_j\right\}$
		
		\vspace{.2cm}
		
		$E_j^*=\left\{\pm y(\beta_2+1):y\in \hat{C}_j\right\}$
	\end{center}
	for all $j=0,\ldots, \Delta_1-1$.
	\begin{itemize}
		\item [Case (i):] If $x_j(\beta_1-1)=\pm x_j^\prime(\beta_1-1)$, for $x_j\neq x_j^\prime\in C_j$,	
		then $(\beta_1-1)(x_j\pm x_j^\prime)=0$, which it is a contradiction, since $x_j+x_j^\prime\neq0$.
		\item [Case (ii):] If $x_j(\beta_1-1)=\pm x_i(\beta_1-1)$,
		for $x_j\in C_j$ and $x_i\in C_i$, then\\ $(\beta_1-1)(x_j\pm x_i)=0$, which it is a contradiction, since  $C_j\cap C_i=\emptyset$, for $i\neq j\in\{0,\ldots,\Delta_1-1\}$ and $x_j+x_i\neq0$.
		\item [Case (iii):] If $y_j(\beta_2+1)=\pm y_j^\prime(\beta_2+1)$, for $y_j\neq y_j^\prime\in \hat{C}_j$, then $(\beta_2+1)(y_j\pm y^\prime_j)=0$, which it is a contradiction, since $y_j+y^\prime_j\neq0$.
		\item [Case (iv):] If $y_j(\beta_2+1)=\pm y_i(\beta_2+1)$, for $y_j\in \hat{C}_j$ and $y_i\in \hat{C}_i$, then\\ $(\beta_2+1)(y_j\pm y_i)=0$, which it is a contradiction, since  $\hat{C}_j\cap \hat{C}_i=\emptyset$, for $i\neq j\in\{0,\ldots,\Delta_1-1\}$ and $y_j+y_i\neq0$.
		\item [Case (v):] As $(\beta_1-1)(\beta_2+1)\in NQR(q)$ then $x_j(\beta_1-1)\neq\pm y_i(\beta_2+1)$, for $x_j\in C_j$ and $y_i\in \hat{C}_i$, since either $x_j(\beta_1-1)\in NQR(q)$ and $y_i(\beta_2+1)\in QR(q)$ or $x_j(\beta_1-1)\in QR(q)$ and $y_i(\beta_2+1)\in NQR(q)$.
		
	\end{itemize}
	
	Now, we shall prove that 
	$\left|\displaystyle\bigcup_{j=0}^{\Delta_1-1}(P_j\cup P_j^*)\right|=\Delta t$, where
	\begin{center}
		$P_j=\left\{x(\beta_1+1):x\in C_j\right\}$
		
		$P_j^*=\left\{-y(\beta_2-1):y\in \hat{C}_j\right\}$
	\end{center}
	for all $j=0,\ldots,\Delta_1-1$.
	
	\begin{itemize}
		\item [Case (i):] If $x_j(\beta_1+1)=x_j^\prime(\beta_1+1)$, for $x_j\neq x_j^\prime\in C_j$, then $(\beta_1+1)(x_j-x_i)=0$, which it is a contradiction.
		\item [Case (ii):] If $x_j(\beta_1+1)=x_i(\beta_1+1)$, for $x_j\in C_j$ and $x_i\in C_i$, then\\ $(\beta_1+1)(x_j-x_i)=0$, which it is a contradiction, since $C_i\cap C_j=\emptyset$, for $i\neq j\in\{0,\ldots,\Delta_1-1\}$.
		\item [Case (iii):] If $-y_j(\beta_2-1)=-y^\prime_j(\beta_2-1)$, for $y_j\neq y_j^\prime\in \hat{C}_j$, then $(\beta_2-1)(y_j-y^\prime_j)=0$, which it is a contradiction.
		\item [Case (iv):] If $-y_j(\beta_2-1)=-y_i(\beta_2-1)$, for $y_j\in \hat{C}_j$ and $y_i\in \hat{C}_i$, then\\ $(\beta_2-1)(y_j-y_i)=0$, which it is a contradiction, since $\hat{C}_j\cap \hat{C}_i\emptyset$, for $i\neq j\in\{0,\ldots,\Delta_1-1\}$.
		\item [Case (v):] As $(\beta_1+1)(\beta_2-1)\in NQR(q)$, then $x_j(\beta_1+1)\neq -y_i(\beta_2-1)$, for $x_j\in C_j$ and $y_i\in \hat{C}_i$, since either $x_j(\beta_1+1)\in NQR(q)$ and $-y_i(\beta_2-1)\in QR(q)$ or $x_j(\beta_1+1)\in QR(q)$ and $-y_i(\beta_2-1)\in NQR(q)$.		
	\end{itemize}

To end,	it is not difficult to prove that $\beta_1\neq\beta_2$, since if $\beta_1=-\beta_1x$, for some $x\in C_0$, then $\beta_1(1+x)=0$, which is a contradiction, since $-1\in \hat{C}_0$.
\end{proof}

\begin{coro}\label{coro:betas}
If $S(\beta_1,\beta_2)$ is a two-quotient strong starter for $\mathbb{F}_q$ given by Theorem \ref{thm:main}, then $S(\beta_2,\beta_1)$, $S(-\beta_1,-\beta_2)$ and $S(-\beta_2,-\beta_1)$ are
two-quotient strong starters for $\mathbb{F}_q$ different from $S(\beta_1,\beta_2)$.	
\end{coro}

Let $q=4t+1$ be a prime power with an odd integer greater than 1,  $C_0\subseteq\mathbb{F}_q^*$ be the subgroup of order $t$ and $C_0,C_1,C_2,C_3$ be the multiplicative cosets of $C_0$. In \cite{dinitz1984} was proven then the following set
$$\hat{S}(a_0,a_1)=\{\{x,a_0x\},\{y,a_1y\}:x\in C_0^{a_0},y\in \hat{C}_1^{a_1}\},$$where $C_0^{a_0}=1/(a_0-1)C_0$ and $C_1^{a_1}=1/(a_1-1)C_1$, is a two-quotient strong starter for $\mathbb{F}_q$. Hence if $\beta_1\in NQR(q)$ and $\beta_2\in\beta_1\hat{C}_0$ are such that $(\beta_1-1)(\beta_2+1)\in NQR(q)$, $(\beta_1+1)(\beta_2-1)\in NQR(q)$ and $(\beta_0-1)\in C_0$ and $-(\beta_1-1)\in \hat{C}_0$, then
\begin{eqnarray*}
\hat{S}(\beta_0,-\beta_1)&=&\{\{x,\beta_0x\},\{y,-\beta_1y\}:x\in C_0^{\beta_0},y\in\hat{C}_0^{\beta_1}\}\\
&=&\{\{x,\beta_0x\},\{y,-\beta_1y\}:x\in C_0,y\in \hat{C}_0\}\\
&=&S(\beta_0,\beta_1).
\end{eqnarray*} 
\subsection{Examples}
In this subsection we give examples of strong starters for $\mathbb{F}(29)$ and $\mathbb{F}(41)$ given by Theorem \ref{thm:main} (main theorem).

Let $\mathbb{F}_{29}=\mathbb{Z}_{29}$, then $k=2$, $t=7$, $\Delta=2$, $\Delta_1=0$ and $\alpha=4$. We have
\begin{eqnarray*}
	QR(29)&=&\{1,4,5,6,7,9,13,16,20,22,23,24,25,28\}\mbox{ and}\\
	NQR(29)&=&\{2,3,8,10,11,12,14,15,17,18,19,21,26,27\}.
\end{eqnarray*}
Hence $C_0=\{1,7,16,20,23,24,25\}$ and $\hat{C}_0=\{4,5,6,9,13,22,28\}$. If $\beta_1=2$ and $\beta_2=26\in\beta_1\hat{C}_0=\{2,10,11,15,17,21,26\}$, then $(\beta_1-1)(\beta_2+1)\in NQR(29)$ and $(\beta_1+1)(\beta_2-1)\in NQR(29)$. Therefore $S(2,26)=S(2,26)_0$, where
\begin{eqnarray*}
	S(2,26)_0&=&\{\{16,3\},\{13,10\}\}\cup\{\{24,19\},\{5,15\}\}\cup\{\{7,14\},\{22,8\}\}\\
	&\cup&\{\{25,21\},\{4,12\}\}\cup\{\{23,17\},\{6,18\}\}\\
	&\cup&\{\{20,11\},\{9,27\}\}\cup\{\{1,2\},\{28,26\}\}
\end{eqnarray*}
is a two-quotient strong starter for $\mathbb{F}_{29}$. Moreover, by Corollary \ref{coro:betas}, we see that $S(26,2)$, $S(27,3)$ and $S(3,27)$ are two-quotient strong starters in $\mathbb{F}_{29}$ different from $S(2,26)$: 
\begin{itemize}
	\item $S(26,2)=S(26,1)_0$, where
	\begin{eqnarray*}
		S(26,1)_0&=&\{\{16,10\},\{13,3\}\}\cup\{\{24,15\},\{5,19\}\}\cup\{\{7,8\},\{22,14\}\}\\
		&\cup&\{\{25,12\},\{4,21\}\}\cup\{\{23,18\},\{6,17\}\}\\
		&\cup&\{\{20,27\},\{9,11\}\}\cup\{\{1,26\},\{28,2\}\}
	\end{eqnarray*}	
	\item $S(27,3)=S(27,3)_0$, where
	\begin{eqnarray*}
		S(27,3)_0&=&\{\{16,26\},\{13,19\}\}\cup\{\{24,10\},\{5,14\}\}\cup\{\{7,15\},\{22,21\}\}\\
		&\cup&\{\{25,8\},\{4,17\}\}\cup\{\{23,12\},\{6,11\}\}\\
		&\cup&\{\{20,18\},\{9,2\}\}\cup\{\{1,27\},\{28,3\}\}\\
	\end{eqnarray*}	
	\item $S(3,27)=S(3,27)_0$, where
	\begin{eqnarray*}
		S(3,27)_0&=&\{\{16,19\},\{13,26\}\}\cup\{\{24,14\},\{5,10\}\}\cup\{\{7,21\},\{22,15\}\}\\
		&\cup&\{\{25,17\},\{4,8\}\}\cup\{\{23,11\},\{6,12\}\}\\
		&\cup&\{\{20,2\},\{9,18\}\}\cup\{\{1,3\},\{28,27\}\}\\
	\end{eqnarray*}
\end{itemize}
It can be verified that all of the starters in the following table are indeed two-quotient strong starters for $\mathbb{F}_{29}$:

\begin{table}[h!]
	\centering 
	\begin{tabular}{|c c c c|} 
		\hline
		$S(\beta_1,\beta_2)$ &	$S(\beta_2,\beta_1)$ & $S(-\beta_1,-\beta_2)$ & $S(-\beta_2,-\beta_1)$ \\ 
		[0.5ex] 
		\hline 
		$S(2,26)$ & $S(26,2)$ & $S(27,3)$  & $S(3,27)$ \\ 
		$S(2,10)$ & $S(10,2)$ & $S(27,19)$ & $S(19,27)$ \\
		$S(3,15)$ & $S(15,3)$ & $S(26,14)$ & $S(14,26)$ \\ 
		$S(3,12)$ & $S(12,3)$ & $S(26,17)$ & $S(17,26)$ \\
		$S(8,11)$ & $S(11,8)$ & $S(21,18)$ & $S(18,21)$ \\
		$S(10,14)$& $S(14,10)$& $S(19,15)$ & $S(15,19)$ \\
		$S(10,17)$& $S(17,10)$& $S(19,12)$ & $S(12,19)$\\
		\hline 
	\end{tabular}
\end{table}	 

As a second example, let $\mathbb{F}_{41}=\mathbb{Z}_{41}$, then $k=3$, $t=5$, $\Delta=4$, $\Delta_1=2$ and  $\alpha=36$. We have
\begin{eqnarray*}
	QR(41)&=&\{1,2,4,5,8,9,10,16,18,20,21,23,25,31,32,33,36,37,39,40\},\mbox{ and}\\
	NQR(41)&=&\{3,6,7,11,12,13,14,15,17,19,22,24,26,27,28,29,30,34,35,38\}.
\end{eqnarray*}
Hence $C_0=\{10,18,16,37,1\}$, $\hat{C}_0=\{31,23,25,4,40\}$, 
$C_1=\{32,33,2,20,36\}$ and $\hat{C}_1=\{9,8,39,21,5\}$.  $\beta_1=3$ then $\beta_2=12\in\beta_1\hat{C}_0=\{11,28,34,12,38\}$ is such that $(\beta_1-1)(\beta_2+1)\in NQR(41)$ and $(\beta_1+1)(\beta_2-1)\in NQR(41)$.Therefore $S(3,12)=S(3,12)_0\cup S(3,12)_1$, where
\begin{eqnarray*}
S(3,12)_0&=&\{\{10,30\},\{31,38\}\}\cup\{\{18,13\},\{23,11\}\}\cup\{\{16,7\},\{25,28\}\}\\
&\cup&\{\{37,29\},\{4,34\}\}\cup\{\{1,3\},\{40,12\}\}\\
S(3,12)_1&=&\{\{32,14\},\{9,15\}\}\cup\{\{33,17\},\{8,27\}\},\{\{2,6\},\{39,24\}\}\\
&\cup&\{\{20,19\},\{21,35\}\}\cup\{\{36,26\},\{5,22\}\}
\end{eqnarray*}	
is a two-quotient strong starter for $\mathbb{F}_{41}$.
Moreover, by Corollary 4.8, we see that $S(12,3)$, $S(38,29)$ and $S(29,38)$ are two-quotient strong starters in $\mathbb{F}_{41}$ different from $S(3,12)$:
\begin{itemize}
	\item $S(12,3)=S(12,3)_0\cup S(12,3)_1$, where
	\begin{eqnarray*}
		S(12,3)_0&=&\{\{10,38\},\{31,30\}\}\cup\{\{18,11\},\{23,13\}\}\cup\{\{16,28\},\{25,7\}\}\\
		&\cup&\{\{37,34\},\{4,29\}\}\cup\{\{1,12\},\{40,3\}\}\\
		S(12,3)_1&=&\{\{32,15\},\{9,14\}\}\cup\{\{33,27\},\{8,17\}\}\cup\{\{2,24\},\{39,6\}\}\\
		&\cup&\{\{20,35\},\{21,19\}\}\cup\{\{36,22\},\{5,26\}\}\\
	\end{eqnarray*}	
	\item $S(38,29)=S(38,29)_0\cup S(38,29)_1$, where
	\begin{eqnarray*}
		S(38,29)_0&=&\{\{10,11\},\{31,3\}\}\cup\{\{18,28\},\{23,30\}\}\cup\{\{16,34\},\{25,13\}\}\\
		&\cup&\{\{37,12\},\{4,7\}\}\cup\{\{1,38\},\{40,29\}\}\\
		S(38,29)_1&=&\{\{32,27\},\{9,26\}\}\cup\{\{33,24\},\{8,14\}\}\cup\{\{2,35\},\{39,17\}\}\\
		&\cup&\{\{20,22\},\{21,6\}\}\cup\{\{36,15\},\{5,19\}\}\\
	\end{eqnarray*}	
	\item $S(29,38)=S(29,38)_0\cup S(29,38)_1$, where
	\begin{eqnarray*}
		S(28,39)_0&=&\{\{10,3\},\{31,11\}\}\cup\{\{18,30\},\{23,28\}\}\cup\{\{16,13\},\{25,34\}\}\\
		&\cup&\{\{37,7\},\{4,12\}\}\cup\{\{1,29\},\{40,38\}\}\\
		S(28,39)_1&=&\{\{32,26\},\{9,27\}\}\cup\{\{33,14\},\{8,24\}\}\cup\{\{2,17\},\{39,35\}\}\\
		&\cup&\{\{20,6\},\{21,22\}\}\cup\{\{36,19\},\{5,15\}\}\\
	\end{eqnarray*}			
\end{itemize}
It can be verified that all of the starters in the following table are indeed two-quotient strong starters for $\mathbb{F}_{41}$:
\begin{table}[h!]
	\centering 
	\begin{tabular}{|c c c c|} 
		\hline
		$S(\beta_1,\beta_2)$ &	$S(\beta_2,\beta_1)$ & $S(-\beta_1,-\beta_2)$ & $S(-\beta_2,-\beta_1)$ \\ 
		[0.5ex] 
		\hline 
		$S(3,12)$ & $S(12,3)$ & $S(38,29)$  & $S(29,38)$ \\ 
		$S(3,28)$ & $S(28,3)$ & $S(38,13)$  & $S(13,38)$ \\
		$S(14,24)$& $S(24,14)$& $S(27,17)$  & $S(17,27)$ \\
		$S(14,22)$& $S(22,14)$& $S(27,19)$ & $S(19,27)$\\
		\hline 
	\end{tabular}
\end{table}	 

\

{\bf Acknowledgment}

\

The authors thank the referee for many constructive suggestions to improve
this paper.

C. R. research supported in part by a CONACyT-M{\' e}xico Postdoctoral fe\-llowship and in part by the National scholarship programme of the Slovak republic. C. A. and A.V. supported by SNI and CONACyT.

\bibliographystyle{amsplain}

\end{document}